\documentclass[a4paper,bibliography=totoc,12pt]{scrartcl} 

\usepackage[english]{babel}
\usepackage[utf8]{inputenc}
\usepackage[T1]{fontenc}
\usepackage{lmodern}
\usepackage{textcomp}
\usepackage{amsfonts}
\usepackage{mathtools}
\usepackage{amssymb}
\usepackage[thmmarks,hyperref,amsmath]{ntheorem}
\usepackage{needspace}
\usepackage{csquotes}
\usepackage{siunitx}
\usepackage[shortlabels]{enumitem}
\usepackage[pdftex]{hyperref}
\usepackage{url}
\usepackage[
	backend=biber,
	style=numeric,
	firstinits = true,
	sorting = nyt,
	maxnames = 8,
	sortlocale = en_US,
	doi=true,
	isbn=false,
	date=year
]{biblatex}
\AtBeginBibliography{\small}
\DeclareFieldFormat{doi}{%
	\newline
	\mkbibacro{DOI}\addcolon\space
	\ifhyperref
		{\href{http://dx.doi.org/#1}{\nolinkurl{#1}}}
		{\nolinkurl{#1}}}
\DeclareFieldFormat{url}{%
	\newline
	\mkbibacro{URL}\addcolon\space
	\ifhyperref
		{\href{#1}{#1}}
		{\nolinkurl{#1}}}

\addtokomafont{disposition}{\rmfamily}

\setlist{itemsep=2pt}
\setlist{parsep=3pt}

\bibliography{incidencematrices.bib}

\newcommand{\F}{\mathbb{F}}

\newcommand{\qbinom}[3]{\genfrac{[}{]}{0pt}{}{#1}{#2}_{#3}}

\newcommand{\cB}{\mathcal{B}}
\newcommand{\cP}{\mathcal{P}}
\newcommand{\des}{\mathcal{D}}

\DeclareMathOperator{\diag}{diag}

\theoremheaderfont{\bfseries}
\theorembodyfont{\upshape}
\theoremstyle{plain}
\theoremseparator{.}
\theoremsymbol{\ensuremath{\square}}
\newtheorem{theorem}{Theorem}[section]
\newtheorem{lemma}[theorem]{Lemma}

\newtheorem{remark}[theorem]{Remark}

\newtheorem{example}[theorem]{Example}

\theoremheaderfont{\itshape}
\theorembodyfont{\upshape}
\theoremstyle{nonumberplain}
\theoremseparator{.}
\theoremsymbol{\rule{1.2ex}{1.2ex}}
\newtheorem{proof}{Proof}

\newcommand{\varleft}{x}
\newcommand{\Varleft}{X}
\newcommand{\varmid}{y}
\newcommand{\Varmid}{Y}
\newcommand{\varright}{z}
\newcommand{\Varright}{Z}

\begin{document}
\title{Higher incidence matrices and\\tactical decomposition matrices}
\author{
Michael Kiermaier%
\thanks{
University of Bayreuth, Institute for Mathematics, 95440 Bayreuth, Germany
\newline
email:~\texttt{michael.kiermaier@uni-bayreuth.de}
\newline
homepage:~\url{https://mathe2.uni-bayreuth.de/michaelk/}
}
\and
Alfred Wassermann%
\thanks{
University of Bayreuth, Institute for Mathematics, 95440 Bayreuth, Germany
\newline
email:~\texttt{alfred.wassermann@uni-bayreuth.de}
\newline
homepage:~\url{https://www.dmi.uni-bayreuth.de/de/team/wassermann/}
}
}
\dedication{\emph{In memory of Professor Zvonimir Janko}}
\setkomafont{dedication}{\normalsize}
\maketitle

\begin{abstract}
    In 1985, Janko and Tran Van Trung published an algorithm for constructing symmetric designs with prescribed automorphisms.
    This algorithm is based on the equations by Dembowski (1958) for tactical decompositions of point-block incidence matrices.
    In the sequel, the algorithm has been generalized and improved in many articles.

    In parallel, higher incidence matrices have been introduced by Wilson in 1982.
    They have proven useful for obtaining several restrictions on the existence of designs.
    For example, a short proof of the generalized Fisher's inequality makes use of these incidence matrices.

    In this paper, we introduce a unified approach to tactical decompositions and incidence matrices.
    It works for both combinatorial and subspace designs alike.
    As a result, we obtain a generalized Fisher's inequality for tactical decompositions of combinatorial and subspace designs.
    Moreover, our approach is explored for the construction of combinatorial and subspace designs of arbitrary strength.
\end{abstract}

\section{Introduction}
A combinatorial $t$-$(v,k,\lambda)$ design $(V, \des)$ is a set $V$ consisting of $v$ \emph{points} together
with a set $\des$ of
$k$-subsets of $V$ called \emph{blocks} such that each $t$-subset of $V$ is contained in exactly
$\lambda$ blocks. In this paper we assume basic familiarity with $t$-designs as given
in \cite{Beth-Jungnickel-Lenz-1999-DesignTheoryI,Beth-Jungnickel-Lenz-1999-DesignTheoryII} or \cite{handbook}. The $q$-analogs of combinatorial designs are called subspace designs.
For an introduction to subspace designs and a discussion of the mechanism of combinatorial $q$-analogs,
the reader is referred to \cite{Braun-Kiermaier-Wassermann-2018-SignalsCommunTechnol:171-211}.

A fundamental result in design theory is Fisher's inequality, stating that $\#\des \geq \#V$ for any non-empty design with $t\geq 2$.
An elegant proof was given in \cite{Bose1949} involving the point-block incidence matrix of the design.
This approach has been generalized in several directions:
Based on tactical decompositions, Block \cite{Block1967} and independently Kantor \cite{Kantor1969} proved a generalization of Fisher's inequality for the number of point and block orbits under a group action.
Ray-Chaudhuri and Wilson \cite{RayChaudhuri-Wilson-1975-OsakaJM12[3]:737-744} used incidence matrices of $s$-subsets versus blocks for $s\geq 1$.
Cameron \cite{Cameron-1974} studied the same question for subspace designs.
As we will see, our article unifies all these generalizations of the method of Bose.

The use of \emph{tactical decompositions} in design theory has been initiated by Dembowski \cite{Dembowski1958},
see also \cite{Dembowski1968} and Beutelspacher \cite[pp. 210--220]{Beutelspacher1982}.
Dembowski's main interest was to use tactical decompositions to study properties of
symmetric designs. From an algorithmic point of view, tactical decompositions were first used
by Janko and Tran Van Trung \cite{Janko1985} to construct symmetric $(78, 22, 6)$ designs.
Their method was picked up and generalized in numerous papers, see
\cite{Cepulic1994,Crnkovic2005,Krcadinac2011} to name just a few. In \cite{Nakic-Pavcevic-2015}
the use of tactical decompositions has been generalized to subspace designs.

The general approach outlined by Janko and Tran Van Trung is to first enumerate all tactical
decomposition matrices of designs with prescribed automorphisms up to permutations of rows and columns.
For this, Dembowski \cite{Dembowski1958} has given powerful constraints for a matrix to be
a tactical decomposition of the point-block incidence matrix of a $2$-design.
In a second step, all remaining tactical decomposition matrices are
expanded -- if possible -- to point-block incidence matrices of designs.

Compared with the well-known method of Kramer and Mesner \cite{Kramer-Mesner-1976} which also restricts the search space
to designs with prescribed automorphisms, the method of Janko and Tran Van Trung has the advantage that it is not necessary to compute all orbits of $k$-subsets of $V$ and therefore allows the search for $2$-$(v,k,\lambda)$ designs with larger $k$ and smaller automorphism group.
The drawback however is that it does not reduce the search space if the prescribed group of automorphisms is point-transitive, and that it seemed to be restricted to $2$-designs for a long time.

The majority of publications on the construction of designs with tactical decompositions are based on the point-block incidence matrix and
only involve constraints derived from the property of a $2$-design.
The only exceptions we are aware of are the articles
\cite{Krcadinac2011,Krcadinac2014,Nakic2015},
where constraints are given for the tactical decomposition of the point-block incidence matrix of a $t$-design of general strength $t\geq 2$.

In this paper, we consider $t$-designs of any strength $t$ and combine Wilson's \cite{Wilson1982}
equations for higher incidence matrices with tactical decompositions.
General tactical decomposition matrices for the incidences of $e$-subsets vs.~blocks with
$e\in\{0,\ldots,t-1\}$ are introduced, and constraints are derived for any strength
$\leq t$.
As a result we get a unification and generalization of all above mentioned Fisher-like inequalities.
Furthermore, we  explore the capabilities of combining higher tactical decomposition matrices
with the method of Janko and Tran van Trung for the construction of designs with $t> 2$.
While our approach is presented for classical block designs, it can also be adapted to subspace designs in a straightforward way.
The interplay to the equations in \cite{Krcadinac2014} deserves further investigation.

\section{Preliminaries}\label{sect:pre}

\subsection{Combinatorial designs}

It is well known that a $t$-$(v,k,\lambda)$ design $(V, \des)$
is also an $s$-$(v,k,\lambda_s)$ design
for $0\leq s\leq t$ where
\[
    \lambda_s = \lambda\,\frac{\binom{v-s}{t-s}}{\binom{k-s}{t-s}} \,.
\]
In particular, $\lambda_0$ is the number of blocks of the design and $\lambda_1$ is the number of blocks
each point is contained in, which is called the \emph{replication number}.
It is also well known that for a $t$-$(v,k,\lambda)$ design the number of blocks
which contain a given $i$-set of points and are disjoint to a given $j$-set of points is equal to
\[
    \lambda_{i,j} = \lambda\, \frac{\binom{v-i-j}{k-j}}{\binom{v-t}{k-t}}\,,
\]
see e.g. \cite[II.4.2, p. 80]{handbook}.

\subsection{Higher incidence matrices}

The $v\times \lambda_0$ point-block incidence matrix $N$ of a $t$-$(v,k,\lambda)$ design $(V, \des)$ is defined
by
\[
    N_{P, B} = \begin{cases}
        1, & \text{if $P\in B$,} \\
        0, & \text{otherwise}
    \end{cases}
\]
for $P\in V$ and $B\in\des$.
Bose \cite{Bose1949} showed for the point-block incidence matrix
$N$ of a $2$-$(v,k,\lambda)$ design the equation
\begin{equation}
    NN^\top = \lambda_1 I + \lambda(J-I\text) \text{,} \label{eq:bose}
\end{equation}
where $I$ is the $v\times v$ identity matrix and $J$ is the $v\times v$ all-ones matrix.

\medskip
For a general $t$-$(v,k,\lambda)$ design with $t\geq 2$,
the $\binom{v}{e}\times \lambda_0$ higher incidence matrix $N^{(e)}$ for $e\leq k$ is defined by
\[
    N^{(e)}_{E, B} = \begin{cases}
        1, & \text{if $E\subset B$,} \\
        0, & \text{otherwise}
    \end{cases}
\]
for $E\in \binom{V}{e}$ and $B\in\des$.
The $\binom{v}{s}\times \binom{v}{e}$ incidence matrix $W^{(se)}$ between $s$-subsets and \emph{all} $e$-subsets of $V$ is defined by
\[
    W^{(se)}_{S, E} = \begin{cases}
        1, & \text{if $S\subset E$,} \\
        0, & \text{otherwise}
    \end{cases}
\]
for $S\in \binom{V}{s}$ and $E\in \binom{V}{e}$.
Wilson \cite{Wilson1982} showed for $e+f\leq t$ the equation
\begin{equation}
    N^{(e)}\,(N^{(f)})^\top = \sum_{i=0}^{\min\{e,f\}} \lambda_{e+f-i,\, i} (W^{(ie)})^\top\, W^{(if)} \,. \label{eq:wilson}
\end{equation}
Note that $N^{(e)}(N^{(f)})^\top$ contains in the row labeled by the $e$-subset $E$ and in the column labeled
by the $f$-subset $F$ the number of blocks of the design which contain both $E$ and $F$.
It is clear that this number is $\lambda_{e+f -\mu}$ with $\mu = \#(E\cap F)$, i.\,e.
\[
    \big(N^{(e)}\,(N^{(f)})^\top\big)_{E,F} = \lambda_{\#(E\cup F)} \,.
\]
Also in \cite{Wilson1982}, Wilson proved among others the equation
\begin{equation}
    W^{(ie)}\,N^{(e)} = \binom{k-i}{e-i}\,N^{(i)}\qquad \mbox{for } 0\leq i\leq e \leq k\,.\label{eq:wilson2}
\end{equation}
For $e=1$ and $i=0$ equation (\ref{eq:wilson2}) simply states that each block of the design contains $k$ points.

\subsection{Tactical decomposition matrices}
Dembowski~\cite{Dembowski1958,Dembowski1968} studied tactical decompositions of incidence structures from group actions,
see also Beutelspacher \cite[pp. 210--220]{Beutelspacher1982}.

Let $(V,\des)$ be a $2$-$(v,k,\lambda)$ design invariant under some group $G$.
The action of $G$ partitions $V$ into orbits $\cP_1, \ldots,\cP_m$ and $\des$ into orbits $\cB_1,\ldots,\cB_n$.
Let $N$ be the point-block incidence matrix of $(V,\des)$ and for $i\in\{1,\ldots,m\}$ and $j\in\{1,\ldots,n\}$ let $N_{i,j}$ be the submatrix of $N$ whose rows are assigned to the elements $\cP_i$ and whose columns to the elements of $\cB_j$.
Then $N_{i,j}$ has a constant number of ones in each row and a constant number of ones in each column.
Such a decomposition of $N$ into submatrices $N_{i,j}$ is called \emph{tactical}.

If we replace for all $i,j$ the submatrix $N_{i,j}$ by the number of ones in each row we get an $(m\times n)$-matrix $\rho$, and if we replace the submatrix $N_{i,j}$ by the number of ones in each column we get an $(m\times n)$-matrix $\kappa$.
The matrices $\rho$ and $\kappa$ are both called \emph{tactical decomposition matrix}.
In \cite{Dembowski1958} the following properties of $\rho$ and $\kappa$ and the matrices
$P = \diag(\# \cP_i)$ and $B = \diag(\# \cB_i)$ are shown:

\begin{align}
    P \cdot \rho                & = \kappa \cdot  B \label{eq:alltop:old}                                         \\
    \rho\cdot (1,\ldots,1)^\top & = (\lambda_1,\ldots,\lambda_1)^\top \label{eq:rho_rowsum:old}                   \\
    (1,\ldots,1) \cdot \kappa   & = (k,\ldots,k) \label{eq:kappa_colsum:old}                                      \\
    \rho\cdot \kappa            & = (\lambda_1 -\lambda)\cdot I + \lambda\cdot P \cdot J \label{eq:rho_kappa:old}
\end{align}

Janko and Tran Van Trung \cite{Janko1985} and many follow-up constructions are using these four equations to build up all
non-isomorphic tactical decomposition matrices $\rho$ (and $\kappa$), usually row-by-row.
In the next section we adapt Wilson's equations for higher incidence matrices
for tactical decomposition matrices.

\section{Higher tactical decomposition matrices}\label{sect:high}

\subsection{The tactical matrices $R$ and $K$}
We fix a finite set $V$ of size $v$.
For $x\in\{0,\ldots,v\}$, let $\mathfrak{P}_x$ be a partition of the set $\binom{V}{x}$.
The part of $\mathfrak{P}_x$ containing some $X\in\binom{V}{x}$ will be denoted by $[X]$.
We call $(\mathfrak{P}_0,\ldots,\mathfrak{P}_v)$ a \emph{tactical sequence of partitions} on $V$
if for all $x,y\in\{0,\ldots,v\}$ with $x \leq y$ and for all $[X], \mathcal{X}\in \mathfrak{P}_x$ and $[Y], \mathcal{Y}\in \mathfrak{P}_y$, the numbers
\[
    R^{(xy)}_{[X],\mathcal{Y}} = \#\{Y\in\mathcal{Y} \mid X \subseteq Y\}
    \qquad\text{and}\qquad
    K^{(xy)}_{\mathcal{X},[Y]} = \#\{X\in\mathcal{X} \mid X \subseteq Y\}
\]
are well-defined, i.\,e.\ they do not depend on the choice of the representative $X$ of $[X]$ nor of the representative $Y$ of $[Y]$.
In this case, the above defined numbers yield matrices $R^{(xy)}, K^{(xy)} \in\mathbb{Z}^{\mathfrak{P}_x \times \mathfrak{P}_y}$.%
\footnote{Formally, it would be more accurate to write $R^{(x,y)}$ and $K^{(x,y)}$ instead of $R^{(xy)}$ and $K^{(xy)}$.
For simplicity, we omit the separating comma whenever there is no danger of confusion.}
A common source of tactical sequences of partitions are permutation groups $G \leq S_V$, where for all $x\in\{0,\ldots,v\}$ the partition $\mathfrak{P}_x$ is the set of orbits of the induced action of $G$ on $\binom{V}{x}$.
The trivial group $G = \{\operatorname{id}_V\}$ leads to the extreme case that all $\mathfrak{P}_x$ are discrete partitions, i.\,e.\ $\mathfrak{P}_x = \{\{\mathcal{X}\} \mid \mathcal{X}\in\binom{V}{x}\}$.
We will refer to this situation as the \enquote{case $\#G = 1$}, where for all $x\leq y$ the matrix $R^{(xy)} = K^{(xy)}$ equals Wilson's higher incidence matrix $W^{(xy)}$.

In the following, we fix a tactical sequence $(\mathfrak{P}_0,\ldots,\mathfrak{P}_v)$ of partitions on $V$.
The matrices $R^{(xx)}$ and $K^{(xx)}$ are $\#\mathfrak{P}_x \times \#\mathfrak{P}_x$ identity matrices.
The matrices $R^{(0x)}$ and $K^{(0x)}$ are of size $1\times\#\mathfrak{P}_x$, where all entries of $K^{(0x)}$ are $1$, and $R^{(0x)}$ contains the part sizes, i.\,e.\ $R^{(0x)}_{\{\emptyset\},\mathcal{X}} = \#\mathcal{X}$.

\begin{example}\label{ex:v6order3}
	We consider the set $V = \{1,2,3,4,5,6\}$ of size $v = \#V = 6$ and the group $G = \langle (1\,2\,3)\,(4\,5\,6)\rangle$ of order $3$ acting on $V$.
	Abbreviating sets $\{a,b\}$ as $ab$ etc., the orbits of $G$ yield the partitions
	\begin{align*}
		\mathfrak{P}_0 & = \{\emptyset\}\text{,}\\
		\mathfrak{P}_1 & = \big\{\{1,2,3\},\{4,5,6\}\big\}\text{,} \\
		\mathfrak{P}_2 & = \big\{\{12,13,23\},\{14,25,36\},\{15,26,34\},\{16,24,35\},\{45,46,56\}\big\}\text{,}\\
		\mathfrak{P}_3 & = \big\{\{123\},\{456\},\{124,235,136\},\{125,236,134\},\{126,234,135\}, \\
		& \phantom{{}=\big\{} \{145,256,346\},\{146,245,356\},\{156,246,345\}\big\}\text{.}
	\end{align*}
	The resulting matrices $R^{(xy)}$ and $K^{(xy)}$ with $0\leq x \leq y \leq 3$ are%
	\footnote{For the assignment of the elements of the partitions $\mathfrak{P}_x$ and $\mathfrak{P}_y$ to the rows and columns, the partitions are ordered as listed above.}
	\begin{align*}
		R^{(00)} & = (1) &
		K^{(00)} & = (1) \\
		R^{(01)} & = \begin{pmatrix}3 & 3\end{pmatrix} &
		K^{(01)} & = \begin{pmatrix}1 & 1\end{pmatrix} \\
		R^{(11)} & = \begin{pmatrix}1 & 0\\0 & 1\end{pmatrix} &
		K^{(11)} & = \begin{pmatrix}1 & 0\\0 & 1\end{pmatrix} \\
		R^{(02)} & = \begin{pmatrix}3 & 3 & 3 & 3 & 3\end{pmatrix} &
		K^{(02)} & = \begin{pmatrix}1 & 1 & 1 & 1 & 1\end{pmatrix} \\
		R^{(12)} & = \begin{pmatrix}2 & 1 & 1 & 1 & 0\\0 & 1 & 1 & 1 & 2\end{pmatrix} &
		K^{(12)} & = \begin{pmatrix}2 & 1 & 1 & 1 & 0\\0 & 1 & 1 & 1 & 2\end{pmatrix} \\
		R^{(22)} & = I_5 &
		K^{(22)} & = I_5 \\
		R^{(03)} & = \begin{pmatrix}1 & 1 & 3 & 3 & 3 & 3 & 3 & 3\end{pmatrix} &
		K^{(03)} & = \begin{pmatrix}1 & 1 & 1 & 1 & 1 & 1 & 1 & 1\end{pmatrix} \\
		R^{(13)} & = \begin{pmatrix}1 & 0 & 2 & 2 & 2 & 1 & 1 & 1\\0 & 1 & 1 & 1 & 1 & 2 & 2 & 2\end{pmatrix} &
		K^{(13)} & = \begin{pmatrix}3 & 0 & 2 & 2 & 2 & 1 & 1 & 1\\0 & 3 & 1 & 1 & 1 & 2 & 2 & 2\end{pmatrix} \\
		R^{(23)} & = \begin{pmatrix}1 & 0 & 1 & 1 & 1 & 0 & 0 & 0\\0 & 0 & 1 & 1 & 0 & 1 & 1 & 0\\0 & 0 & 0 & 1 & 1 & 1 & 0 & 1\\0 & 0 & 1 & 0 & 1 & 0 & 1 & 1\\0 & 1 & 0 & 0 & 0 & 1 & 1 & 1\end{pmatrix} &
		K^{(23)} & = \begin{pmatrix}3 & 0 & 1 & 1 & 1 & 0 & 0 & 0\\0 & 0 & 1 & 1 & 0 & 1 & 1 & 0\\0 & 0 & 0 & 1 & 1 & 1 & 0 & 1\\0 & 0 & 1 & 0 & 1 & 0 & 1 & 1\\0 & 3 & 0 & 0 & 0 & 1 & 1 & 1\end{pmatrix} \\
		R^{(33)} & = I_8 &
		K^{(33)} & = I_8
	\end{align*}
\end{example}

As suggested by the example, the following lemma shows that matrices $R^{(xy)}$ and $K^{(xy)}$ determine each other.
For $x\in\{0,\ldots,v\}$ we define $D^{(x)}\in\mathbb{Z}^{\mathfrak{P}_x \times \mathfrak{P}_x}$ as the invertible diagonal matrix
with entries $D^{(x)}_{\mathcal{X},\mathcal{X}} = \#\mathcal{X}$.

\begin{lemma}\label{lem:alltop}
    Let $x,y\in\{0,\ldots,v\}$ be integers with $x \leq y$.
    Then for all $\mathcal{X}\in\mathfrak{P}_x$ and all $\mathcal{Y}\in\mathfrak{P}_y$ we have
    \[
	    \#\mathcal{X} \cdot R^{(xy)}_{\mathcal{X},\mathcal{Y}} = \#\mathcal{Y} \cdot K^{(xy)}_{\mathcal{X},\mathcal{Y}}\text{.}
    \]
    This can be rewritten as the equality of matrix products
    \[
	    D^{(x)}\, R^{(xy)} = K^{(xy)}\, D^{(y)}\text{.}
    \]
\end{lemma}

\begin{proof}
    Count the set $\{(X,Y)\in \mathcal{X}\times\mathcal{Y} \mid X \subseteq Y\}$ in two ways.
\end{proof}

\begin{lemma}\label{lem:RK_transitive}
    Let $x,y,z\in\{0,\ldots,v\}$ be integers with $x\leq y\leq z$.
    Then
    \[
        R^{(xy)} \, R^{(yz)} = \binom{z-x}{y-x}\, R^{(xz)}
        \qquad\text{and}\qquad
        K^{(xy)} \, K^{(yz)} = \binom{z-x}{y-x}\, K^{(xz)}\text{.}
    \]
\end{lemma}

\begin{proof}
    We fix $[X] = \mathcal{X}\in\mathfrak{P}_x$ and $\mathcal{Z}\in\mathfrak{P}_z$.
    By counting the set
    \[
        A = \big\{(Y,Z)\in \tbinom{V}{y} \times \mathcal{Z} \mid X \subseteq Y \subseteq Z\big\}
    \]
    in two ways, we get that
    \[
        \big(R^{(xy)}\,R^{(yz)}\big)_{\mathcal{X},\mathcal{Z}}
        = \sum_{\mathcal{Y}\in \mathfrak{P}_y} R^{(xy)}_{\mathcal{X},\mathcal{Y}}\, R^{(yz)}_{\mathcal{Y},\mathcal{Z}}
        = \#A
        = R^{(xz)}_{\mathcal{X},\mathcal{Z}}\,\binom{z-x}{y-x}\text{.}
    \]
    The equality for the $K$-matrices is shown analogously.
\end{proof}

\begin{lemma}\label{lem:R_neighbors}
	Let $x,y\in\{0,\ldots,v\}$ be integers with $x \leq y$.
	Then
	\begin{align*}
		R^{(xy)} & = \frac{1}{(y-x)!}\, R^{(x,x+1)}\, R^{(x+1,x+2)}\,\cdot\ldots\cdot\, R^{(y-1,y)}\quad\text{and}\\
		K^{(xy)} & = \frac{1}{(y-x)!}\, K^{(x,x+1)}\, K^{(x+1,x+2)}\,\cdot\ldots\cdot\, K^{(y-1,y)}\text{.}
	\end{align*}
\end{lemma}

\begin{proof}
	The statement is true for $x = y$.
	For $x < y$, Lemma~\ref{lem:RK_transitive} gives
	\[
		R^{(xy)} = \frac{1}{y-x}\, R^{(x,x+1)} R^{(x+1,y)}\text{,}
		\]
	which inductively equals
	\begin{multline*}
		\frac{1}{y-x}\, R^{(x,x+1)}\, \cdot \frac{1}{(y-x-1)!}\, R^{(x+1,x+2)}\cdot\ldots\cdot R^{(y-1,y)} \\
		= \frac{1}{(y-x)!}\, R^{(x,x+1)}\cdot\ldots\cdot R^{(y-1,y)}\text{.}
	\end{multline*}
	The statement for $K^{(xy)}$ is shown in the same way.
\end{proof}

\begin{remark}\label{rem:R_chain}
Lemma~\ref{lem:R_neighbors} shows that for any non-negative integer $s$, the chain of matrices $R^{(01)}, R^{(12)},\ldots, R^{(s-1,s)}$ determines all the matrices $R^{(xy)}$ with integers $x,y$ and $0 \leq x \leq y \leq s$.
Moreover, as the diagonal of the diagonal matrix $D^{(x)}$ equals the row vector $R^{(0x)}$, by Lemma~\ref{lem:alltop} also the matrices $K^{(xy)}$ with $0 \leq x \leq y \leq s$ are determined.
\end{remark}

\begin{lemma}\label{lem:KR}
    Let $x,y,z\in\{0,\ldots,v\}$.
    \begin{enumerate}[(a)]
        \item\label{lem:KR:KtR} For $\varmid \leq \min(\varleft,\varright)$, the entries of the matrix $(K^{(\varmid\varleft)})^\top \, R^{(\varmid\varright)}\in \mathbb{\Varright}^{\mathfrak{P}_\varleft\times\mathfrak{P}_\varright}$ are given by
        \[
            \big((K^{(\varmid\varleft)})^{\top}\, R^{(\varmid\varright)}\big)_{[\Varleft], \mathcal{\Varright}}
            = \sum_{\Varright\in\mathcal{\Varright}}\binom{\#(\Varleft\cap \Varright)}{\varmid}\text{.}
        \]
        \item\label{lem:KR:RKt} For $\max(\varleft,\varright) \leq \varmid$, the entries of the matrix $R^{(\varleft\varmid)}\, (K^{(\varright\varmid)})^\top$ are given by
        \[
            \big(R^{(\varleft\varmid)}\, (K^{(\varright\varmid)})^\top\big)_{[\Varleft],\mathcal{\Varright}}
            = \sum_{\Varright\in\mathcal{\Varright}} \binom{v - \#(\Varleft \cup \Varright)}{v - \varmid}\text{.}
        \]
    \end{enumerate}
\end{lemma}

\begin{proof}
    For part~\ref{lem:KR:KtR}, we fix two parts $[\Varleft] = \mathcal{\Varleft} \in \mathfrak{P}_\varleft$ and $\mathcal{\Varright}\in\mathfrak{P}_\varright$.
    By counting the set
    \[
        A = \big\{ (\Varmid,\Varright)\in \tbinom{V}{\varmid} \times \mathcal{\Varright} \mid \Varmid \subseteq \Varleft\cap \Varright\big\}
    \]
    in two ways, we get that
    \begin{multline*}
        \big((K^{(\varmid\varleft)})^\top\, R^{(\varmid\varright)}\big)_{\mathcal{\Varleft},\mathcal{\Varright}}
        = \sum_{\mathcal{\Varmid}\in\mathfrak{P}_\varmid} K^{(\varmid\varleft)}_{\mathcal{\Varmid},\mathcal{\Varleft}} R^{(\varmid\varright)}_{\mathcal{\Varmid},\mathcal{\Varright}} \\
        = \#A
        = \sum_{\Varright\in\mathcal{\Varright}} \#\{\Varmid\in \tbinom{V}{\varmid} \mid \Varmid \subseteq \Varleft\cap \Varright\}
        = \sum_{\Varright\in\mathcal{\Varright}} \binom{\#(\Varleft \cap \Varright)}{\varmid}\text{.}
    \end{multline*}

    Part~\ref{lem:KR:RKt} is shown similarly by counting $\{(\Varmid,\Varright)\in \tbinom{V}{\varmid}\times\mathcal{\Varright} \mid \Varleft \cup \Varright \subseteq \Varmid\}$ in two ways.
    Alternatively, it can be derived from part~\ref{lem:KR:KtR} by dualization, or
    it will follow as a special case of the later Lemma~\ref{lem:wilson_lhs}.
\end{proof}

\begin{lemma}\label{lem:RK_constant_sum}
    Let $x,y\in\{0,\ldots,v\}$ be integers with $x \leq y$.
    Then the matrix $K^{(xy)}$ has constant column sum $\binom{y}{x}$, and the matrix $R^{(xy)}$ has constant row sum $\binom{v-x}{v-y}$.
\end{lemma}

\begin{proof}
    Lemma~\ref{lem:RK_transitive} with $x \leftarrow 0$, $y \leftarrow x$ and $z \leftarrow y$ gives
    \[
	    (1\, \ldots\, 1)\cdot K^{(xy)} = \binom{x}{y} \cdot (1\,\ldots\,1)\text{.}
    \]
    Lemma~\ref{lem:KR}\ref{lem:KR:RKt} with $z = 0$ gives
    \[
	    R^{(xz)} \cdot (1\,\ldots\,1)^\top = \binom{v-x}{v-z} \cdot (1\,\ldots\,1)^\top\text{.}
    \]
\end{proof}

\subsection{The tactical decomposition matrices $\rho^{(x)}$ and $\kappa^{(x)}$}
Now we fix a non-empty $t$-$(v,k,\lambda)$ design $(V,\des)$ with $t \leq k \leq v-t$, such that the block set is the union of parts in $\mathfrak{P}_k$, i.\,e.\ $\des = \bigcup\mathfrak{B}$ with $\mathfrak{B} \subseteq\mathfrak{P}_k$.
Its numbers $\lambda_{i,j}$ and $\lambda_i = \lambda_{i,0}$ are defined as in Section \ref{sect:pre}.
Since $\des$ is non-empty, $\lambda_{i,j} > 0$ for all admissible $i,j$, i.\,e.\ for all non-negative integers $i,j$ with $i+j \leq t$.

For $x\in\{0,\ldots,k\}$ we define the tactical decomposition matrices $\rho^{(x)}, \kappa^{(x)} \in \mathbb{Z}^{\mathfrak{P}_x\times \mathfrak{B}}$ via
\[
    \rho^{(x)}_{[X], \mathcal{B}} = \#\{B\in \mathcal{B} \mid X \subseteq B\}
    \qquad\text{and}\qquad
    \kappa^{(x)}_{\mathcal{X}, [B]} = \#\{X\in \mathcal{X} \mid X \subseteq B\}\text{.}
\]
By the properties of the fixed tactical sequence $(\mathfrak{P}_0,\ldots,\mathfrak{P}_v)$ of partitions on $V$, this definition does not depend on the choice of the representatives.
Note that $\rho^{(x)}$ is the restriction of $R^{(xk)}$ to the columns whose labels are contained in $\mathfrak{B}$.
The other way round, $R^{(xy)}$ is the matrix $\rho^{(x)}$ of the complete design $\binom{V}{y}$.
A similar note holds for the $\kappa$- and $K$-matrices.
In particular, $\rho^{(0)}$ and $\kappa^{(0)}$ are of size $1\times\#\mathfrak{B}$, where all entries of $\kappa^{(0)}$ are $1$, and $\rho^{(0)}$ contains the sizes of the block parts.

The next lemma says that the matrices $\rho^{(x)}$ and $\kappa^{(x)}$ determine each other.
It is a direct consequence of Lemma~\ref{lem:alltop}.
We define $\delta\in\mathbb{Z}^{\mathfrak{B}\times\mathfrak{B}}$ as the diagonal matrix with the entries $\delta_{\mathcal{B},\mathcal{B}} = \#\mathcal{B}$.

\begin{lemma}\label{lem:alltop_design}
	Let $x\in\{0,\ldots,k\}$
	Then for all $\mathcal{X}\in\mathfrak{P}_x$ and all $\mathcal{B}\in\mathfrak{B}$ we have
	\[
		\#\mathcal{X} \cdot \rho^{(x)}_{\mathcal{X},\mathcal{B}} = \#\mathcal{B} \cdot \kappa^{(x)}_{\mathcal{X},\mathcal{B}}\text{.}
	\]
	This can be rewritten as the equality of matrix products
	\[
		D^{(x)}\, \rho^{(x)} = \kappa^{(x)}\, \delta\text{.}
	\]
\end{lemma}
In the special case $x = 1$, Lemma~\ref{lem:alltop_design} recovers Equation~\eqref{eq:alltop:old}.

\begin{example}\label{ex:v6order3_rho_kappa}
	We continue with Example~\ref{ex:v6order3} and consider $G$-invariant $2$-$(6,3,2)$ designs.
	Its $\lambda_{ij}$-values are displayed in the following triangle.
	\[
		\begin{array}{ccccc}
			& &  \lambda_{0,0} = 10  \\
			& \lambda_{1,0} = 5 & & \lambda_{0,1} = 5 \\
			\lambda_{2,0} = 2 & & \lambda_{1,1} = 3 & & \lambda_{0,2} = 2
		\end{array}
	\]
	There are several ways of forming unions of the orbits in $\mathfrak{P}_3$ to get a $G$-invariant $2$-$(6,3,2)$ design $(V,\des)$.
	We consider the one given by the $1$st, $3$rd, $6$th and $8$th orbit (in the ordering of Example~\ref{ex:v6order3}), that is $\des = \bigcup\mathfrak{B}$ with
	\[
		\mathfrak{B} = \big\{\{123\},\{124,235,136\},\{145,256,346\},\{156,246,345\}\big\}\text{.}
	\]
	The restriction of the matrices $R^{(x3)}$ and $K^{(x3)}$ to the corresponding columns yields the following tactical decomposition matrices $\rho^{(x)}$ and $\kappa^{(x)}$ of $(V,\des)$.
	\begin{align*}
		\rho^{(0)} & = \begin{pmatrix}1 & 3 & 3 & 3\end{pmatrix} &
		\kappa^{(0)} & = \begin{pmatrix}1 & 1 & 1 & 1\end{pmatrix} \\
		\rho^{(1)} & = \begin{pmatrix}1 & 2 & 1 & 1\\0 & 1 & 2 & 2\end{pmatrix} &
		\kappa^{(1)} & = \begin{pmatrix}3 & 2 & 1 & 1\\0 & 1 & 2 & 2\end{pmatrix} \\
		\rho^{(2)} & = \begin{pmatrix}1 & 1 & 0 & 0\\0 & 1 & 1 & 0\\0 & 0 & 1 & 1\\0 & 1 & 0 & 1\\0 & 0 & 1 & 1\end{pmatrix} &
		\kappa^{(2)} & = \begin{pmatrix}3 & 1 & 0 & 0\\0 & 1 & 1 & 0\\0 & 0 & 1 & 1\\0 & 1 & 0 & 1\\0 & 0 & 1 & 1\end{pmatrix}
	\end{align*}
\end{example}

The following lemma shows that the matrices $\rho^{(x)}$ and $\kappa^{(x)}$ with lower values of $x$ are determined by the higher ones.
\begin{lemma}\label{lem:wilson_2nd}
    Let $x,y$ be non-negative integers with $x \leq y \leq k$.
    Then
    \[
        R^{(xy)}\,\rho^{(y)} = \binom{k-x}{y-x} \, \rho^{(x)}
        \qquad\text{and}\qquad
        K^{(xy)}\,\kappa^{(y)} = \binom{k-x}{y-x}\kappa^{(x)}\text{.}
    \]
\end{lemma}

\begin{proof}
    By Lemma~\ref{lem:RK_transitive}, $R^{(xy)}\, R^{(yk)} = \binom{k-x}{y-x}\, R^{(xk)}$ and $K^{(xy)}\, K^{(yk)} = \binom{k-x}{y-x}\, K^{(xk)}$.
    The restriction of these matrices to the columns belonging to the elements of $\mathfrak{B}$ gives the statement of the lemma.
\end{proof}

As in Lemma~\ref{lem:RK_constant_sum}, for $x = 0$ the $K$-equation in Lemma~\ref{lem:wilson_2nd} states that the sum of each column of $\kappa^{(y)}$ equals $\binom{k}{y}$.
In the special case $x = 0$ and $y = 1$, we get back Equation~\eqref{eq:kappa_colsum:old}.

\begin{lemma}\label{lem:wilson_lhs}
    Let $e,f$ be non-negative integers with $e + f \leq t$.
    Then the entries of the matrix $\rho^{(e)}\cdot(\kappa^{(f)})^\top \in \mathbb{Z}^{\mathfrak{P}_e\times\mathfrak{P}_f}$ are given by
    \[
        \big(\rho^{(e)}\, (\kappa^{(f)})^\top\big)_{[E],\mathcal{F}}
        = \sum_{F\in\mathcal{F}} \lambda_{\#(E \cup F)}\text{.}
    \]
    In particular, the matrix $\rho^{(e)}\,(\kappa^{(f)})^\top$ only depends on the tactical sequence $(\mathfrak{P}_0,\ldots,\mathfrak{P}_v)$ of partitions and the parameters of the design, but not on the specific choice of the blocks.
\end{lemma}

\begin{proof}
    We fix two parts $[E] = \mathcal{E} \in\mathfrak{P}_e$ and $\mathcal{F}\in\mathfrak{P}_f$.
    Counting the set
    \[
        A = \{ (B,F)\in \des\times \mathcal{F} \mid E \cup F \subseteq B\}
    \]
    in two ways, we see that
    \begin{multline*}
        \big(\rho^{(e)}\, (\kappa^{(f)})^\top\big)_{\mathcal{E},\mathcal{F}}
        = \sum_{\mathcal{B}\in\mathfrak{B}} \rho^{(e)}_{\mathcal{E},\mathcal{B}}\, \kappa^{(f)}_{\mathcal{F},\mathcal{B}} \\
        = \#A
        = \sum_{F\in\mathcal{F}} \#\{B\in \des \mid E \cup F \subseteq B\}
        = \sum_{F\in\mathcal{F}} \lambda_{\#(E \cup F)}\text{.}
    \end{multline*}
\end{proof}

For the complete design $\des = \binom{V}{k}$, Lemma~\ref{lem:wilson_lhs} reduces to Lemma~\ref{lem:KR}\ref{lem:KR:RKt}.
In the special case $f = 0$, Lemma~\ref{lem:wilson_lhs} states that the sum of each row of $\rho^{(e)}$ equals $\lambda_e$.
In the special case $e = f = 0$, Lemma~\ref{lem:wilson_lhs} yields the formula $\sum_{\mathcal{B}\in\mathfrak{B}} \#\mathcal{B} = \lambda_0$, which says that the sizes of the block parts add up to the size of the design.
In the special case $e = 1$ and $f = 0$, we recover Equation~\eqref{eq:rho_rowsum:old}.

The following lemma is contained in \cite[proof of Prop.~1]{Wilson1982}, essentially.

\begin{lemma}\label{lem:pascal-formula}
    Let $x,y$ be non-negative integers with $x + y \leq t$.
    Then
    \[
        \lambda_x = \sum_{j=0}^y \lambda_{x+j,\, y-j} \binom{y}{j} = \sum_{j=0}^y \lambda_{x+y-j,\,j}\binom{y}{j}\text{.}
    \]
\end{lemma}

\begin{proof}
    By $x + y \leq t$ there exist disjoint subsets $X$ and $Y$ of $V$ of size $x$ and $y$.
    Double counting the set $\{(B,J) \in \des \times 2^Y \mid X \subseteq B \text{ and } B \cap Y = J\}$ gives the stated formula.
\end{proof}

\begin{theorem}\label{thm:wilson}
    Let $V$ be a finite set of size $v$ and let $(\mathfrak{P}_0,\ldots,\mathfrak{P}_v)$ be a tactical sequence of partitions on $V$.
    Let $(V,\des)$ be a non-empty $t$-$(v,k,\lambda)$ design with $t\leq k \leq v-t$, such that the block set has the form $\des = \bigcup\mathfrak{B}$ with $\mathfrak{B} \subseteq \mathfrak{P}_k$.
    Let $e,f$ be non-negative integers with $e + f \leq t$.

    Then
    \[
        \rho^{(e)}\, (\kappa^{(f)})^\top
        = \sum_{j=0}^{\min(e,f)} \lambda_{e + f - j,\, j} \,(K^{(je)})^\top R^{(jf)}\text{.}
    \]
\end{theorem}

\begin{proof}
    We fix two parts $[E] = \mathcal{E} \in\mathfrak{P}_e$ and $\mathcal{F}\in\mathfrak{P}_f$.
    By Lemma~\ref{lem:wilson_lhs}
    \[
        \big(\rho^{(e)}\, (\kappa^{(f)})^\top\big)_{\mathcal{E},\mathcal{F}}
        = \sum_{F\in\mathcal{F}} \lambda_{\#(E \cup F)}\text{.}
    \]
    By Lemma~\ref{lem:pascal-formula} (with $x = \#(E \cup F)$ and $y = \#(E \cap F)$; note that $x + y = e + f \leq t$) and Lemma~\ref{lem:KR}\ref{lem:KR:KtR}, this expression equals
    \[
        \sum_{F\in\mathcal{F}} \sum_{j=0}^{\min(e,f)} \lambda_{e+f-j,\, j} \binom{\#(E \cap F)}{j}
        = \sum_{j=0}^{\min(e,f)} \lambda_{e + f - j,\, j} \,\big((K^{(je)})^\top R^{(jf)}\big)_{\mathcal{E},\mathcal{F}}\text{.}
    \]
\end{proof}

In the special case $e = f = 1$, Theorem~\ref{thm:wilson} gives Equation~\eqref{eq:rho_kappa:old}.

\subsection{The averaged matrices $W^{(xy)}$ and $\omega^{(x)}$}
In the case $\#G = 1$, all parts are of size $1$ and hence $R^{(xy)} = K^{(xy)}$ for all $x \leq y$ and $\rho^{(x)} = \kappa^{(x)}$ for all $x$.
To mimic that situation, we introduce \enquote{averaged} versions of $R^{(xy)}$ and $K^{(xy)}$ and of $\rho^{(x)}$ and $\kappa^{(y)}$, based on the transformation formulas in Lemma~\ref{lem:alltop} and Lemma~\ref{lem:alltop_design}.

As all diagonal matrices $D^{(x)}$ and $\delta$ have positive diagonal entries, it makes sense to write $\sqrt{D^{(x)}}$ and $\sqrt{\delta}$, where the diagonal entries are replaced by their (positive) square roots.
Clearly, all the matrices $D^{(x)}$, $\sqrt{D^{(x)}}$, $\delta$ and $\sqrt{\delta}$ are invertible.
Now for integers $x,y$ with $0 \leq x \leq y\leq v$ we define
\[
	W^{(xy)}
	= \sqrt{D^{(x)}}\, R^{(xy)}\, \sqrt{D^{(y)}}^{-1}
	= \sqrt{D^{(x)}}^{-1}\, K^{(xy)}\, \sqrt{D^{(y)}}
\]
and for an integer $x$ with $0\leq x \leq k$ we define
\[
	\omega^{(x)}
	= \sqrt{D^{(x)}}\, \rho^{(x)}\, \sqrt{\delta}^{-1}
	= \sqrt{D^{(x)}}^{-1}\, \kappa^{(x)}\, \sqrt{\delta}
\]
Again, $\omega^{(x)}$ is the restriction of the matrix $W^{(xk)}$ to the columns whose labels are contained in $\mathfrak{B}$.
In the case $\#G = 1$, $W^{(xy)}$ and $\omega^{(x)}$ equal Wilson's $W$- and $N$-matrices in \cite{Wilson1982}.

\begin{example}\label{ex:v6order3_omega}
	In Example~\ref{ex:v6order3_rho_kappa}, we have
	\begin{align*}
		\omega^{(0)} & = \begin{pmatrix}1 & \sqrt{3} & \sqrt{3} & \sqrt{3}\end{pmatrix}\text{,} \\
		\omega^{(1)} & = \begin{pmatrix}\sqrt{3} & 2 & 1 & 1\\0 & 1 & 2 & 2\end{pmatrix}\text{,} \\
		\omega^{(2)} & = \begin{pmatrix}\sqrt{3} & 1 & 0 & 0\\0 & 1 & 1 & 0\\0 & 0 & 1 & 1\\0 & 1 & 0 & 1\\0 & 0 & 1 & 1\end{pmatrix}\text{.}
	\end{align*}
\end{example}

Our above results can be transformed into formulas for the averaged matrices.

\begin{lemma}\label{lem:W_transitive}
	Let $x,y,z\in\{0,\ldots,v\}$ be integers with $x\leq y \leq z$.
	Then
	\[
		W^{(xy)}\, W^{(yz)} = \binom{z-x}{y-x}\, W^{(xz)}\text{.}
	\]
\end{lemma}

\begin{proof}
	By Lemma~\ref{lem:RK_transitive} we have $R^{(xy)}\, R^{(yz)} = \binom{z-x}{y-x}\, R^{(xz)}$.
	Therefore
	\[
		\sqrt{D^{(x)}}\, R^{(xy)}\, \sqrt{D^{(y)}}^{-1} \cdot \sqrt{D^{(y)}}\, R^{(yz)}\, \sqrt{D^{(z)}}^{-1} = \tbinom{z-x}{y-x}\,\sqrt{D^{(x)}}\, R^{(xz)}\, \sqrt{D^{(z)}}^{-1}\text{,}
	\]
	which gives the claimed statement.
\end{proof}

\begin{lemma}
	Let $x,y\in\{0,\ldots,k\}$ with $x \leq y$.
	Then
	\[
		W^{(xy)}\, \omega^{(y)} = \binom{k-x}{y-x}\,\omega^{(x)}\text{.}
	\]
\end{lemma}

\begin{proof}
	Lemma~\ref{lem:W_transitive} gives $W^{(xy)}\, W^{(yk)} = \binom{k-x}{y-x}\, W^{(xk)}$.
	Now we restrict this equation to the columns belonging to the elements in $\mathfrak{B}$.
\end{proof}

\begin{theorem}\label{thm:wilson_omega}
	Let $V$ be a finite set of size $v$ and let $(\mathfrak{P}_0,\ldots,\mathfrak{P}_v)$ be a tactical sequence of partitions on $V$.
	Let $(V,\des)$ be a non-empty $t$-$(v,k,\lambda)$ design with $t\leq k \leq v-t$, such that the block set has the form $\des = \bigcup\mathfrak{B}$ with $\mathfrak{B} \subseteq \mathfrak{P}_k$.
	Let $e,f$ be non-negative integers with $e + f \leq t$.

	Then
	\[
		\omega^{(e)}\,(\omega^{(f)})^\top = \sum_{j=0}^{\min(e,f)}\lambda_{e + f - j,\, j}\,(W^{(je)})^\top\, W^{(jf)}\text{.}
	\]
\end{theorem}

\begin{proof}
    Left-multiplication by $\sqrt{D^{(e)}}$ and right-multiplication by $\sqrt{D^{(f)}}^{-1}$ of the formula in
    Theorem~\ref{thm:wilson} lead to
    \begin{multline*}
    	\sqrt{D^{(e)}}\,\rho^{(e)}\,\sqrt{\delta}^{-1}\cdot\sqrt{\delta}\,(\rho^{(f)})^\top\, \sqrt{D^{(f)}}^{-1} \\
        = \sum_{j=0}^{\min(e,f)} \lambda_{e + f - j,\, j}\cdot\sqrt{D^{(e)}}\,(K^{(je)})^\top\, \sqrt{D^{(j)}}^{-1}\cdot \sqrt{D^{(j)}}\, R^{(jf)}\,\sqrt{D^{(f)}}^{-1}\text{,}
    \end{multline*}
	which gives the claimed statement.
\end{proof}

Theorem~\ref{thm:wilson_omega} is the literal generalization of \cite[Prop.~1]{Wilson1982}.
It allows us to further follow this famous paper investigating the definiteness of the
symmetric matrix $\omega^{(x)}\,(\omega^{(x)})^\top$.

\begin{lemma}\label{lem:wilson_omega}
	Let $x$ be an integer with $0 \leq 2x \leq t$.
	Then the matrix $\omega^{(x)}\, (\omega^{(x)})^\top$ is positive definite.
\end{lemma}

\begin{proof}
Theorem~\ref{thm:wilson_omega} with $x = e = f$ states that
\[
	\omega^{(x)}\, (\omega^{(x)})^\top
        = \sum_{j=0}^x \lambda_{2x - j,\, j} \,(W^{(jx)})^\top\, W^{(jx)}\text{.}
\]
All the matrices $\lambda_{2x - j, j}\,(W^{(jx)})^\top\, W^{(jx)}$ are positive semidefinite; note that $\lambda_{2x - j,\, j} > 0$ because of $(2x - j) + j = 2x \leq t$.
For $j = x$, the matrix $\lambda_{2x - j, j}\,(W^{(jx)})^\top\, W^{(jx)} = \lambda_{x,x} I_x^\top I_x$ is positive definite.
Now being the sum of a positive definite matrix and positive semidefinite matrices, $\omega^{(x)}\,(\omega^{(x)})^\top$ is positive definite.
\end{proof}

We remark that the condition $2x \leq t$ in Lemma~\ref{lem:wilson_omega} cannot be dropped.
In Example~\ref{ex:v6order3_omega}, the matrix $\omega^{(2)}$ is of size $5\times 4$, such that the $(5\times 5)$-matrix $\omega^{(2)}\, (\omega^{(2)})^\top$ cannot be regular.

\begin{theorem}\label{thm:raychaudhuri_wilson}
	Let $V$ be a finite set of size $v$ and let $(\mathfrak{P}_0,\ldots,\mathfrak{P}_v)$ be a tactical sequence of partitions on $V$.
	Let $(V,\des)$ be a non-empty $t$-$(v,k,\lambda)$ design with $t\leq k \leq v-t$, such that the block set has the form $\des = \bigcup\mathfrak{B}$ with $\mathfrak{B} \subseteq \mathfrak{P}_k$.

	Then $\#\mathfrak{B} \geq \#\mathfrak{P}_{x}$ for all $x\in\{0,\ldots,\lfloor t/2 \rfloor\}$.
\end{theorem}

\begin{proof}
	Let $x\in\{0,\ldots,\lfloor t/2 \rfloor\}$.
	By Lemma~\ref{lem:wilson_omega}, the $(\#\mathfrak{P}_x\times\#\mathfrak{P}_x)$-matrix $\omega^{(x)}\, (\omega^{(x)})^\top$ is positive definite and hence regular.
	Therefore the rank of the $(\#\mathfrak{P}_x\times\#\mathfrak{B})$-matrix $\omega^{(x)}$ is $\#\mathfrak{P}_x$.
	So $\#\mathfrak{P}_x \leq \#\mathfrak{B}$.
\end{proof}

In the case $\#G = 1$, Theorem~\ref{thm:raychaudhuri_wilson} specializes to the generalized Fisher's inequality \cite[Thm.~1]{RayChaudhuri-Wilson-1975-OsakaJM12[3]:737-744},
see also \cite[Thm.~1]{Wilson1982}, which again generalizes the ordinary Fisher's inequality (which is the case $\#G = 1$ and $t = 2$).
Moreover, in the case $t = 2$ Theorem~\ref{thm:raychaudhuri_wilson} specializes to Block's theorem, see \cite[Cor.~2.2]{Block1967} and \cite[Thm.~4.1]{Kantor1969}.

\section{Generalization to subspace designs}
The definition of a block design only involves terms within the subset lattice of the set $V$ of finite size $v = \#V$.
Replacing this subset lattice (and the derived notions) by the subspace lattice of a $\F_q$-vector space of finite dimension $v$, we get a $q$-analog of this definition.
The resulting object is known as a $t$-$(v,k,\lambda)_q$ \emph{subspace design} $(V,\des)$, which therefore is defined as a set $\des$ of $k$-dimensional subspaces of $V$ such that each $t$-dimensional subspace $T$ of $V$ is contained in exactly $\lambda$ elements of $\des$.
Up to a certain point, the theory of subspace designs closely matches the theory of block designs, where the latter may be understood as the limit case $q = 1$, see \cite{Braun-Kiermaier-Wassermann-2018-SignalsCommunTechnol:171-211}.

We decided to restrict the presentation of Section~\ref{sect:high} to the classical setting of block designs for better readability.
But all the results and proofs are still true for subspace designs, in their natural $q$-analog counterpart.
This means that we can replace the subset lattice by the subspace lattice, implying that the cardinality is replaced by the dimension, binomial coefficients are replaced by Gaussian binomial coefficients, the set union is replaced by the sum of subspaces, the permutation group $S_V$ is replaced by the group $\operatorname{P\Gamma L}(V)$, etc.

It's indicated to add a few words on the counterpart of the numbers $\lambda_{i,j}$.
There are two natural $q$-analogs, which coincide in the case of block designs.
For a $t$-$(v,k,\lambda)_q$ subspace design on a $v$-dimensional $\F_q$-vector space $V$, we fix non-negative integers $i$ and $j$ with $i + j \leq t$ and an $i$-dimensional subspace $I$ of $V$.

In the first variant (which is the one discussed in \cite{Braun-Kiermaier-Wassermann-2018-SignalsCommunTechnol:171-211}), a $(v-j)$-dimensional subspace $J$ with $I \leq J\leq V$ is fixed, and it is shown that the number
\[
    \lambda^{(1)}_{i,j} = \#\{B\in \des \mid I \leq B\leq J\} = \qbinom{v-i-j}{k-i}{q} / \qbinom{v-t}{k-t}{q} \cdot \lambda
\]
does not depend on the choice of $I$ and $J$.
This first variant is the more natural one in the sense that $\lambda^{(1)}_{i,j}$ captures the $\lambda$-values of all the reduced, derived, residual and dual designs of $(V,\des)$ including their iterations and combinations \cite{Kiermaier-Laue-2015-AiMoC9[1]:105-115}.

For a direct counterpart of Lemma~\ref{lem:pascal-formula} we need the second variant \cite{Suzuki1990, Kiermaier-Pavcevic-2015-JCD23[11]:463-480}, though.
Now a $j$-dimensional subspace $J$ of $V$ having trivial intersection with $I$ is fixed, and it is shown that the number
\begin{align*}
    \lambda^{(2)}_{i,j} & = \#\{B\in \des \mid I \leq B\text{ and }J \cap B = \{\boldsymbol{0}\}\}    \\
                        & = q^{j(k-i)} \qbinom{v-i-j}{k-i}{q} / \qbinom{v-t}{k-t}{q} \cdot \lambda \\
                        & = q^{j(k-i)} \lambda^{(1)}_{i,j}
\end{align*}
does not depend on the choice of $I$ and $J$.
For the proof of Lemma~\ref{lem:pascal-formula} we note that the counted set can be rewritten as
\[
    \{B\in \des \mid I \leq B\text{ and }J \cap B = \{\boldsymbol{0}\}\}
    = \{B \in \des \mid B \cap (I+J) = I\}\text{,}
\]
where the inclusion \enquote{$\supseteq$} is clear and the inclusion \enquote{$\subseteq$} follows from the dimension formula
\begin{align*}
     & \!\!\!\!\!\!\!\!\!\!\dim(B \cap (I+J))                               \\
     & = \dim((B\cap (I+J)) \cap J) + \dim((B\cap (I+J)) + J) - \dim(J)     \\
     & \leq \dim(B\cap J) + \dim(I+J) - \dim(J) = 0 + (i+j) - j = i\text{.}
\end{align*}

In the case $e = f = 1$, the subspace design version of Theorem~\ref{thm:wilson} is contained in \cite[Thm.~2]{Nakic-Pavcevic-2015}.
In the case $\#G = 1$, the subspace design version of Theorem~\ref{thm:raychaudhuri_wilson} is contained in \cite[statement~(3')]{Cameron-1974}, see also \cite[Thm.~2.3]{Suzuki1990}.

\section{Algorithmic use}\label{sect:alg}
The practical use of higher tactical decomposition matrices for computer construction has yet to be explored.
We conclude this paper by applying higher tactical decomposition matrices to the small design parameters $3$-$(10,4,1)$, which have also been investigated in \cite[Sec.~2]{Krcadinac2011}.
The corresponding $\lambda_{ij}$-values are displayed in the following triangle.
\[
	\begin{array}{ccccccc}
		& & &  \lambda_{0,0} = 30  \\
		& & \lambda_{1,0} = 12 & & \lambda_{0,1} = 18 \\
		& \lambda_{2,0} = 4 & & \lambda_{1,1} = 8 & & \lambda_{0,2} = 10 \\
		\lambda_{3,0} = 1 & & \lambda_{2,1} = 3 & & \lambda_{1,2} = 5 & & \lambda_{0,3} = 5
	\end{array}
\]
We fix a group $G$ of order $3$ acting on the $10$-element set $V$ with exactly one fixed point.
Now we consider $G$-invariant $3$-$(10, 4, 1)$ designs with exactly three fixed blocks, i.\,e.\ we fix (up to a permutation of the columns)
\[
	\rho^{(0)} = \begin{pmatrix}1&1&1&3&3&3&3&3&3&3&3&3\end{pmatrix}\text{.}
\]
According to \cite{Krcadinac2011}, up to isomorphism there are
eight tactical decomposition matrices $\rho^{(1)}$ fulfilling the
equations (\ref{eq:alltop:old})--(\ref{eq:rho_kappa:old}), thereby corresponding to designs with the reduced parameters $2$-$(10,4,4$).
The extended method of \cite{Krcadinac2011} then shows that exactly one of these eight matrices leads to a $G$-invariant $3$-$(10,4,1)$ design.

As a proof of concept, we test if the restrictions given by Section \ref{sect:high} lead to the same result.
The choice of $V=\{0,\ldots, 9\}$ and
\[
	G=\langle (1\, 2\, 3)\,(4\, 5\, 6)\,(7\, 8\, 9)\rangle
\]
results in the partitions
\begin{align*}
	\mathfrak{P}_1 & = \big\{\{0\},\{1,2,3\},\{4,5,6\},\{7,8,9\}\big\}\text{,} \\
	\mathfrak{P}_2 & = \big\{\{01,02,03\},\{04,05,06\},\{07,08,09\},\{12,13,23\},\{14,25,36\},\\
	& \phantom{{}=\big\{} \{15,26,34\},\{16,24,35\},\{17,28,39\},\{18,29,37\},\{19,27,38\},\\
	& \phantom{{}=\big\{} \{45,46,56\},\{47,58,69\},\{48,59,67\},\{49,57,68\},\{78,79,89\}\big\}
\end{align*}
and the $R$- and $K$-matrices%
\footnote{Again, for the assignment to the columns and rows, the partitions $\mathfrak{P}_x$ are assumed to be ordered as above.}
\begin{align*}
R^{(01)} & = \begin{pmatrix} 1&3&3&3 \end{pmatrix}\text{,} \\
K^{(01)} & = \begin{pmatrix} 1&1&1&1 \end{pmatrix}\text{,} \\
R^{(02)} & = \begin{pmatrix} 3&3&3&3&3&3&3&3&3&3&3&3&3&3&3 \end{pmatrix}\text{,} \\
K^{(02)} & = \begin{pmatrix} 1&1&1&1&1&1&1&1&1&1&1&1&1&1&1 \end{pmatrix}\text{,} \\
R^{(12)} & = \begin{pmatrix}
3&3&3&0&0&0&0&0&0&0&0&0&0&0&0 \\
1&0&0&2&1&1&1&1&1&1&0&0&0&0&0 \\
0&1&0&0&1&1&1&0&0&0&2&1&1&1&0 \\
0&0&1&0&0&0&0&1&1&1&0&1&1&1&2 \\
\end{pmatrix}\text{ and} \\
K^{(12)} & = \begin{pmatrix}
1&1&1&0&0&0&0&0&0&0&0&0&0&0&0 \\
1&0&0&2&1&1&1&1&1&1&0&0&0&0&0 \\
0&1&0&0&1&1&1&0&0&0&2&1&1&1&0 \\
0&0&1&0&0&0&0&1&1&1&0&1&1&1&2 \\
\end{pmatrix}\text{.}
\end{align*}
First, we reproduced the following $8$ representatives for $\rho^{(1)}$.
\allowdisplaybreaks
\begin{align*}
    \rho^{(1)}_1 & = \begin{pmatrix}
        1&1&1&0&3&3&0&0&3&0&0&0\\
        0&0&0&2&1&1&2&1&2&1&2&0\\
        1&0&1&1&1&1&1&2&0&2&1&1\\
        0&1&0&1&1&1&1&1&1&1&1&3\\
    \end{pmatrix}\\
    \rho^{(1)}_2 & = \begin{pmatrix}
        1&1&1&0&3&3&0&0&3&0&0&0\\
        0&0&0&2&1&1&2&1&2&1&2&0\\
        1&0&1&1&1&1&1&1&0&1&2&2\\
        0&1&0&1&1&1&1&2&1&2&0&2\\
    \end{pmatrix}\\
    \rho^{(1)}_3 & = \begin{pmatrix}
        1&1&1&0&3&3&0&0&3&0&0&0\\
        0&0&0&2&1&1&2&1&2&1&2&0\\
        1&0&0&1&1&2&1&1&0&2&1&2\\
        0&1&1&1&1&0&1&2&1&1&1&2\\
    \end{pmatrix}\\
    \rho^{(1)}_4 & = \begin{pmatrix}
        1&1&1&0&3&3&0&0&3&0&0&0\\
        0&0&0&3&1&1&1&1&2&1&1&1\\
        1&0&1&1&1&1&2&1&0&1&1&2\\
        0&1&0&0&1&1&1&2&1&2&2&1\\
    \end{pmatrix}\\
    \rho^{(1)}_5 & = \begin{pmatrix}
        1&1&1&0&3&3&0&0&3&0&0&0\\
        0&0&0&1&2&2&2&1&0&1&1&2\\
        1&0&1&1&0&1&1&2&1&1&2&1\\
        0&1&0&2&1&0&1&1&2&2&1&1\\
    \end{pmatrix}\\
    \rho^{(1)}_6 & = \begin{pmatrix}
        1&1&1&0&3&3&0&0&3&0&0&0\\
        0&0&1&1&2&1&2&1&0&1&1&2\\
        0&1&0&1&1&1&0&2&1&2&2&1\\
        1&0&0&2&0&1&2&1&2&1&1&1\\
    \end{pmatrix}\\
    \rho^{(1)}_7 & = \begin{pmatrix}
        1&1&1&0&3&3&0&0&3&0&0&0\\
        0&0&1&1&2&1&2&1&0&1&1&2\\
        0&1&0&1&1&0&1&2&2&1&2&1\\
        1&0&0&2&0&2&1&1&1&2&1&1\\
    \end{pmatrix}\\
    \rho^{(1)}_8 & = \begin{pmatrix}
        1&1&1&0&3&3&0&0&3&0&0&0\\
        0&0&1&2&1&1&2&1&1&1&2&0\\
        1&0&0&2&1&1&1&1&1&2&0&2\\
        0&1&0&0&1&1&1&2&1&1&2&2\\
    \end{pmatrix}
\end{align*}
For each of these eight matrices $\rho^{(1)}$, we determine all possible matrices $\rho^{(2)}$ as the solutions of the matrix equations
\begin{align*}
    R^{(02)} \,   \rho^{(2)}          & = \binom{k}{2}\, \rho^{(0)}                                      \text{,}\\
    R^{(12)} \,   \rho^{(2)}          & = \binom{k-1}{1}\,\rho^{(1)}                                     \text{,}\\
    \rho^{(2)} \, (\kappa^{(0)})^\top & = \lambda_2\cdot (1,\ldots,1)^\top                               \text{ and}\\
    \rho^{(2)} \, (\kappa^{(1)})^\top & = \sum_{j=0}^{1} \lambda_{3 - j,\, j} \,(K^{(j2)})^\top R^{(j1)}
\end{align*}
from Lemma \ref{lem:wilson_2nd} and Theorem \ref{thm:wilson} which give linear restrictions on the entries of the matrix $\rho^{(2)}$.
For the evaluation, we note that $\kappa^{(0)} = (1\; 1\; 1\; 1)$ and that $\kappa^{(1)}$ is determined by Lemma \ref{lem:alltop} by the given $\rho^{(1)}$.
The right hand side of the last equation equals
\[
    \begin{pmatrix}
    4&4&4&1&1&1&1&1&1&1&1&1&1&1&1 \\
    6&3&3&9&6&6&6&6&6&6&3&3&3&3&3 \\
    3&6&3&3&6&6&6&3&3&3&9&6&6&6&3 \\
    3&3&6&3&3&3&3&6&6&6&3&6&6&6&9 \\
    \end{pmatrix}^\top\text{.}
\]
Solving this system of Diophantine linear equations with the second author's software \emph{solvediophant} \cite{Wassermann2021} shows that -- just as in \cite{Krcadinac2011} -- exactly one of the eight matrices $\rho^{(1)}$ (namely $\rho^{(1)}_8$) can be extended to a higher decomposition matrix $\rho^{(2)}$.
One of the $47\,040$ solutions is
\[
\rho^{(2)} =
\begin{pmatrix}
    0&0&1&0&1&1&0&0&1&0&0&0\\
    1&0&0&0&1&1&0&0&1&0&0&0\\
    0&1&0&0&1&1&0&0&1&0&0&0\\
    0&0&1&1&0&0&1&0&0&0&1&0\\
    0&0&0&2&0&0&0&1&1&0&0&0\\
    0&0&0&1&1&0&1&0&0&1&0&0\\
    0&0&0&1&0&1&1&0&0&1&0&0\\
    0&0&0&0&0&1&1&1&0&0&1&0\\
    0&0&0&0&1&0&1&1&0&0&1&0\\
    0&0&0&0&0&0&0&0&1&1&2&0\\
    1&0&0&1&0&0&0&0&0&1&0&1\\
    0&0&0&0&1&0&0&1&0&1&0&1\\
    0&0&0&0&0&1&0&1&0&1&0&1\\
    0&0&0&0&0&0&1&0&1&0&0&2\\
    0&1&0&0&0&0&0&1&0&0&1&1\\
\end{pmatrix}\text{.}
\]

What remains to do is the so-called \emph{indexing step}, assigning to each column of $\rho^{(2)}$ a suitable element of $\mathfrak{P}_4$.
In this case, it is almost trivial to read off the (unique) design belonging to $\rho^{(2)}$ based on the partitions $\mathfrak{P}_1$ and $\mathfrak{P}_2$.
We give two examples.

The fourth column belongs to a set $\mathcal{B}$ of blocks of size $\#\mathcal{B} = \rho^{(0)}_4 = 3$.
The fifth row of $\rho^{(2)}$ is assigned to the part $\{14,25,36\}$ of $\mathfrak{P}_2$.
So by $\rho^{(2)}_{5,4} = 2$, two blocks in $\mathcal{B}$ contain $\{1,4\}$, two blocks contain $\{2,5\}$ and two blocks contain $\{3,6\}$.
The only possibility is $\mathcal{B} = \{1245,1346,2356\}$.

The fifth column belongs to a set $\mathcal{B}$ of blocks of size $\#\mathcal{B} = \rho^{(0)}_5 = 3$.
By $\rho^{(1)}_{2,5} = 1$, there is a unique block $B\in\mathcal{B}$ containing $1$.
By $\rho^{(1)}_{1,5} = 1$, $0\in B$.
By $\rho^{(2)}_{6,5} = 1$, there is a unique block $B\in\mathcal{B}$ containing $\{1,5\}$.
So $5\in B$.
Similarly, $\rho^{(2)}_{9,5} = 1$ implies $8\in B$.
Hence $B = \{0158\}$ and therefore $\mathcal{B}$ is the $G$-orbit $\{0158,0269,0347\}$.

In this way, we end up with the design given by the partition
\begin{align*}
	\mathfrak{B} & = 
	\big\{
		\{0456\},\{0789\},\{0123\},\\
		& \phantom{{}=\big\{}\{1245,1346,2356\},\{0158,0269,0347\},\{0167,0248,0359\},\\
		& \phantom{{}=\big\{}\{1268,1357,2349\},\{1478,2589,3679\},\{0149,0257,0368\}, \\
		& \phantom{{}=\big\{}\{1569,2467,3458\},\{1279,2378,1389\},\{4579,5678,4689\}
	\big\}\text{.}
\end{align*}

We would like to add a few more remarks.
Note that in general, Lemma~\ref{lem:alltop} will give additional divisibility
conditions on the entries of a matrix $\rho^{(e)}$ depending on the orbit lengths.
However, in this specific example, these are trivial.

The approach to determine the whole matrix $\rho^{(2)}$ by solving a single system of equations
is not yet optimal, since even in this small example there are already
$47\,040$ solutions for $\rho^{(2)}$.
A reduction of isomorphic copies up to permutation of rows and columns still has to be implemented.

It is conceivable that an interleaved row-by-row enumeration strategy of all tactical decomposition matrices
$\rho^{(e)}$, $0\leq e<t$ will allow to counter the combinatorial explosion of solutions by an intermediate
rejection of isomorphic partial tactical decomposition matrices.
In particular, a row-by-row enumeration strategy will be mandatory if the prescribed automorphisms are
point-transitive.

In the ordinary approach only using the decomposition matrices $\rho^{(1)}$ and $\kappa^{(1)}$, the indexing step is usually a non-trivial computational problem.
In our above example the matrix $\rho^{(2)}$ was quite helpful to this end.
Therefore the higher tactical decomposition matrices might also prove useful as an intermediate computational goal for the indexing step.

Finally, we note that for the construction of combinatorial designs or subspace designs with
prescribed automorphisms it is irrelevant if the right hand sides
of equations involving $\rho^{(e)} \, (\kappa^{(f)})^\top$ are determined from Theorem \ref{thm:wilson}
or from Lemma \ref{lem:wilson_lhs}.
But unlike Lemma \ref{lem:wilson_lhs}, the variant in Theorem \ref{thm:wilson} does not require the exact knowledge of the partitions $\mathfrak{P}_i$.
This might be useful if the tactical decomposition does not stem from a prescribed automorphisms, but is determined by prescribing the sizes
of the partition of the blocks of the design (i.\,e.\ $\rho^{(0)}$) together
with a chain of tactical decomposition matrices $R^{(01)},R^{(12)},\ldots,R^{(t-1,t)}$ (see Remark~\ref{rem:R_chain}) such
that each matrix $R^{(i,i+1)}$ has constant row sum equal to $v-i$ (as required by Lemma~\ref{lem:RK_constant_sum}).

\section*{Acknowledgements}
We want to thank the anonymous referee for helpful comments.
The second author wants to thank the organizers of the
\emph{Combinatorial Constructions Workshop} (CCW 2022) in Zagreb
for their hospitality. During this workshop the initial ideas of the paper have been developed.

\sloppy
\printbibliography

\end{document}